\documentclass[12pt,a4paper,english,reqno]{amsart}
\usepackage[a4paper,footskip=1.5em,margin=2cm]{geometry}
\usepackage{amsmath,amssymb,amsthm,mathtools}
\usepackage[mathscr]{euscript}
\usepackage[usenames,dvipsnames]{color}
\usepackage{adjustbox,tikz,calc,graphics,babel,standalone}
\usepackage{subcaption}
\usepackage{csquotes,enumerate,verbatim}
\usepackage[final]{microtype}
\usepackage[numbers]{natbib}
\usepackage{bbm}
\usepackage{tikz}
\usepackage{pgf}
\usetikzlibrary{calc}

\usetikzlibrary{shapes.misc,calc,intersections,patterns,decorations.pathreplacing}
\usepackage{hyperref}
\hypersetup{colorlinks=true,linkcolor=blue,citecolor=blue,pdfpagemode=UseNone,pdfstartview={XYZ null null 1.00}}
\usepackage{cmtiup}
\usepackage{tikz}
\usepackage{pgfplots}
\usepackage{graphicx}
\usetikzlibrary{patterns}
\pgfplotsset{compat=1.11}
\usepgfplotslibrary{fillbetween}
\usetikzlibrary{intersections}
\makeatletter
\newcommand{\Spvek}[2][c]{%
  \gdef\@VORNE{1}
  \left(\hskip-\arraycolsep%
    \begin{array}{#1}\vekSp@lten{#2}\end{array}%
  \hskip-\arraycolsep\right)}

\def\vekSp@lten#1{\xvekSp@lten#1;vekL@stLine;}
\def\vekL@stLine{vekL@stLine}
\def\xvekSp@lten#1;{\def\temp{#1}%
  \ifx\temp\vekL@stLine
  \else
    \ifnum\@VORNE=1\gdef\@VORNE{0}
    \else\@arraycr\fi%
    #1%
    \expandafter\xvekSp@lten
  \fi}
\makeatother

\newcommand{\tvect}[3]{%
   \ensuremath{\Bigl(\negthinspace\begin{smallmatrix}#1\\#2\\#3\end{smallmatrix}\Bigr)}}

\pgfdeclarelayer{bg}
\pgfsetlayers{bg,main}
\tikzstyle{empty}=[circle,draw=black!80,thick]
\tikzstyle{emptyn}=[circle,draw=black!80,fill=white,scale=0.5]
\tikzstyle{nero}=[circle,draw=black!80,fill=black!80,thick]

\theoremstyle{plain}
\newtheorem*{theorem*}{Theorem}
\newtheorem{theorem}{Theorem}[section]
\newtheorem{lemma}[theorem]{Lemma}

\newtheorem{proposition}[theorem]{Proposition}
\newtheorem*{claim*}{Claim}

\newtheorem{question}[theorem]{Question}
\theoremstyle{remark}

\newcommand{\fo}{f_{\text{odd}}}

\let\emptyset\varnothing
\let\eps\varepsilon
\let\originalleft\left
\let\originalright\right
\renewcommand{\left}{\mathopen{}\mathclose\bgroup\originalleft}
\renewcommand{\right}{\aftergroup\egroup\originalright}

\begin{document}

\title{Partition problems in high dimensional boxes}

\author{Matija Bucic}
\thanks{The first author was supported in part by SNSF grant 200021-175573.}
\address{Department of Mathematics, ETH, R\"amistrasse 101, 8092 Z\"urich, Switzerland}
\email{\texttt{matija.bucic}@\texttt{math.ethz.ch}}

\author{Bernard Lidick\'{y}}
\thanks{The second author was supported in part by NSF grant DMS-1600390.}
\address{Department of Mathematics, Iowa State University, 396 Carver Hall, Ames, IA 50011, USA}
\email{lidicky@iastate.edu}

\author{Jason Long}
\address{Department of Pure Mathematics and Mathematical Statistics, University of Cambridge, Wilberforce Road, Cambridge CB3\thinspace0WB, UK}
\email{jl694@cam.ac.uk}

\author{Adam Zsolt Wagner}
\address{Department of Mathematics, University of Illinois, 1409 W.\/ Green Street, Urbana IL 61801, USA}
\email{zawagne2@illinois.edu}

\date{}


\begin{abstract}
Alon, Bohman, Holzman and Kleitman proved that any partition of a $d$-dimensional discrete box into proper sub-boxes must consist of at least $2^d$ sub-boxes. Recently,  Leader, Mili\'{c}evi\'{c} and Tan considered the question of how many odd-sized proper boxes are needed to partition a $d$-dimensional box of odd size, and they asked whether the trivial construction consisting of $3^d$ boxes is best possible. We show that approximately $2.93^d$ boxes are enough, and consider some natural generalisations.
\end{abstract}
\maketitle

\section{Introduction}

The following lovely problem, due to Kearnes and Kiss~\cite[][Problem 5.5]{kearneskiss}, was presented at the open problem session at the August 1999 meeting at MIT that was held to celebrate Daniel Kleitman's 65th birthday~\cite{saks}. A set of the form $$A=A_1 \times A_2 \times \ldots \times A_d,$$ where $A_1, A_2,\ldots,A_d$ are finite sets with $|A_i|\geq 2$ will be called here a \emph{$d$-dimensional discrete box}. A set of the form $B=B_1\times B_2\times\ldots\times B_d$, where $B_i\subseteq A_i$ for all $i\in [d]$, is a \emph{sub-box} of $A$. Such a sub-box $B$ is said to be \emph{proper} if $B_i \neq A_i$ for every $i$. The question of Kearnes and Kiss was as follows: can the box $A=A_1 \times A_2 \times \ldots \times A_d$ be partitioned into fewer than $2^d$ proper sub-boxes? 

Within a day, Alon, Bohman, Holzman and Kleitman solved~\cite{alonbohman} this problem. Their eventual distillation of the proof, which we present in Section~\ref{sec:setup}, is a ``proof from the book''.
\begin{theorem}[\label{thm:alon}Alon, Bohman, Holzman, Kleitman~\cite{alonbohman}]
Let $A$ be a $d$-dimensional discrete box, and let $\{B^1,B^2,\ldots,B^m\}$ be a partition of $A$ into proper sub-boxes. Then $m\geq 2^d$.
\end{theorem}

The following interesting question was recently posed by Leader, Mili\'{c}evi\'{c} and Tan~\cite{leader}. Say that the $d$-dimensional box $A=A_1\times A_2\times \ldots \times A_d$ is \emph{odd} if each $|A_i|$ is odd (and finite). Similarly, say that the sub-box $B=B_1\times B_2\times \ldots \times B_d$ is \emph{odd} if $|B_i|$ is odd  for all $i$. It is easy to see that given a $d$-dimensional odd box $A$, there exists a partition of $A$ into $3^d$ odd proper sub-boxes, by partitioning each side into three odd parts and taking all possible products.

\begin{question}[\label{qu:leader}Leader, Mili\'{c}evi\'{c}, Tan~\cite{leader}]
Let $A$ be a $d$-dimensional odd box, and let $\{B^1,B^2,\ldots,B^m\}$ be a partition of $A$ into odd proper sub-boxes. Does it follow that then $m\geq 3^d$?.
\end{question}
Our first result is that the answer to this question is `no':
\begin{theorem}\label{thm:main}
 Let $d\in\mathbb{Z}^+$ be divisible by $3$. Then there exists a partition of $[5]^d$ into  $25^{d/3}\leq 2.93^d$ odd proper sub-boxes. 
\end{theorem}
The proof is based on an example which shows that it is possible to partition $[5]^3$ into $25$ odd proper sub-boxes, see Figure~\ref{fig:25oddbox}. We originally found examples with the help of a computer, but the example presented here was found by hand, keeping in mind certain properties of the examples provided by the computer. The solution is not unique.

\begin{figure}[h]
\caption{25 odd boxes partitioning $[5]^3$}
\input{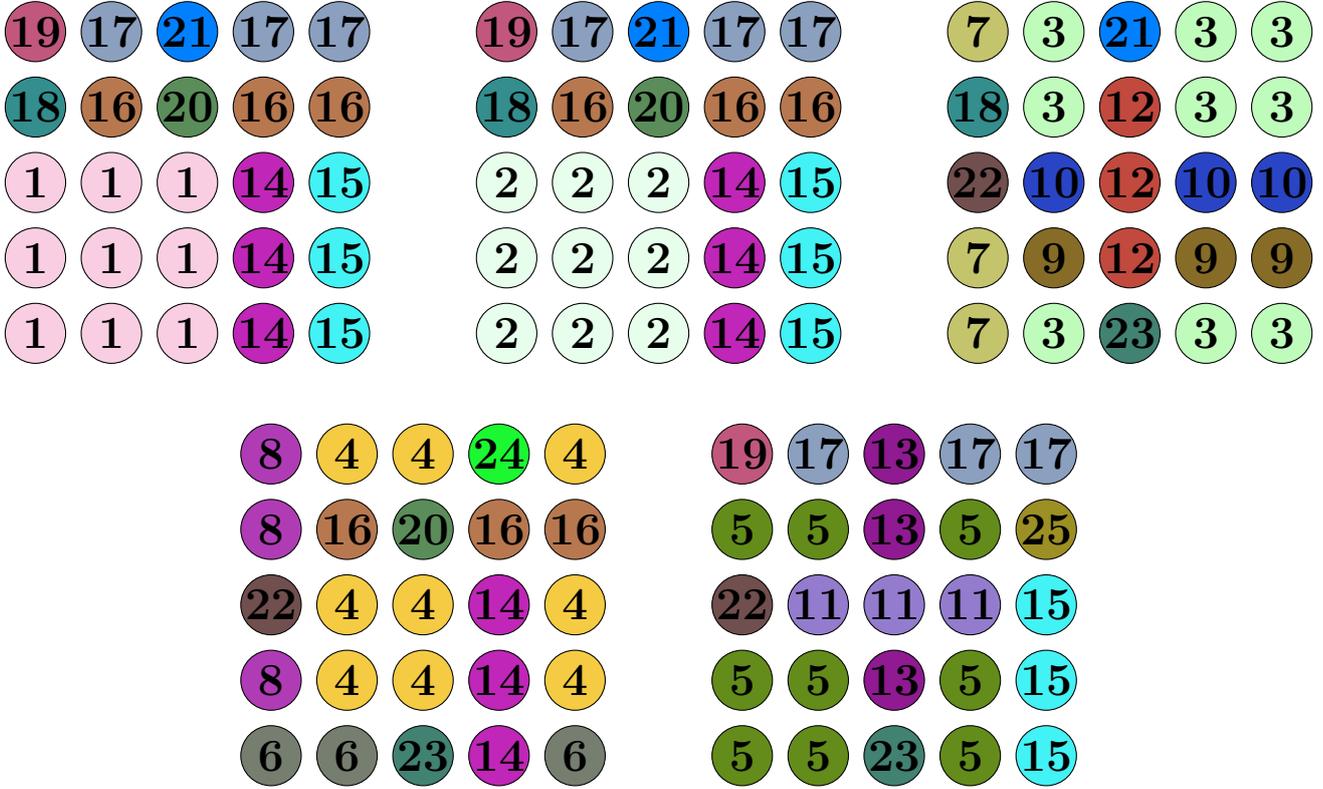}
\centering
\label{fig:25oddbox}
\end{figure}

The situation changes, however if we require the odd boxes in our partition to be products of intervals. Say that the box $B=B_1\times B_2\times\ldots\times B_d$ is a \emph{brick} if for each $i\in\{1,2,\ldots,d\}$ there exist integers $i_0,i_1$ with $i_0\leq i_1$, such that $B_i=\{i_0,i_0+1,\ldots,i_1\}$. As examples, consider the following two boxes:
\begin{itemize}
\item The set $B=\{2,3,4\}\times \{4\}\times \{1,6,7\}$ is an odd proper sub-box of $[7]^3$ but it is not a brick, as $\{1,6,7\}$ does not have the required form.
\item The set $B=\{2,3,4\}\times \{3,4\}$ is a proper brick contained in $[5]^2$. However it is not odd, as $|\{3,4\}|=2$.
\end{itemize}
Our next result shows that the answer to Question~\ref{qu:leader} is `yes' under the additional assumption that the sub-boxes are in fact proper, odd bricks.
\begin{theorem}\label{prop:bricks}
Let $n\geq 2$ be odd, and $d\geq 1$ arbitrary integer. Let $\{B^1,B^2,\ldots,B^m\}$ be a partition of $[n]^d$ into proper, odd bricks. Then $m\geq 3^d$.
\end{theorem}


There are a number of natural generalisations of this question. In this paper we shall consider a weakening of the parity constraint. A key property enforced by a partition into odd, proper boxes is that any axis-parallel line through $[n]^d$ intersects at least 3 distinct sub-boxes, with the result that the most obvious construction involves dividing each dimension into 3 parts and taking the resulting $3^d$ sub-boxes. It is therefore natural to pose the following question, which we refer to as the \emph{$k$-piercing} problem.

\begin{question}[$k$-piercing]\label{qu:kpiercingbox}
Let $n\ge k$ and $d\ge 1$ be integers. Let $\{B^1,B^2,\dots, B^m\}$ be a partition of $[n]^d$ into proper boxes with the property that every axis-parallel line intersects at least $k$ distinct $B_i$ (we call this the $k$-piercing property). How small can $m$ be?
\end{question}

This question can obviously be phrased in a continuous setting, replacing $[n]$ with the interval $[0,1]$ and eliminating $n$ altogether. For simplicity we shall not do this, but instead we will generally present bounds on $m$ as a function of $k$ and $d$ only by considering $n$ large enough (for most of our results it is sufficient to take $n>3k$). 

The 2-piercing problem corresponds precisely to the original problem of Kearnes and Kiss, and so the bound $m\ge 2^d$ holds. However, Theorem~\ref{thm:main} tells us that $m< 3^d$ when $k=3$. In fact the easy observation that $3^d$ cannot be a lower bound follows from a simple 2-dimensional construction shown in Figure~\ref{fig:2dimkp}.

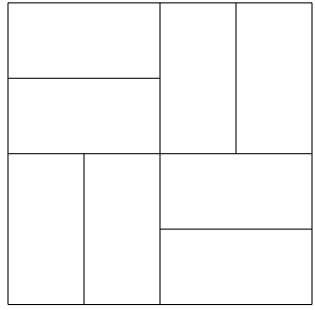
\begin{figure}
\caption{8 bricks in two dimensions satisfying the $3$-piercing property.}
\begin{center}
\begin{tikzpicture}
\draw [-] (0,0) -- (4,0);
\draw [-] (0,0) -- (0,4);
\draw [-] (0,4) -- (4,4);
\draw [-] (4,0) -- (4,4);

\draw [-] (2,0) -- (2,4);
\draw [-] (0,2) -- (4,2);
\draw [-] (1,0) -- (1,2);
\draw [-] (3,2) -- (3,4);
\draw [-] (2,1) -- (4,1);
\draw [-] (0,3) -- (2,3);
\end{tikzpicture}
\end{center}

\label{fig:2dimkp}
\end{figure}

Our later results will concentrate on the $k$-piercing problem. We show, perhaps surprisingly, that $m$ is bounded by $c^dk$ for some $c$ which is independent of $k$. 

\begin{theorem}\label{thm:kpiercingbox}
Let $k \ge 2$ and $d\ge 1$ be integers. For $n$ large enough there exists a partition $\{B^1,\dots,B^m\}$ of $[n]^d$ into proper boxes having the $k$-piercing property with $m\le 15^{d/2}k$.
\end{theorem}

Recall that the answer to Question~\ref{qu:leader} changes fundamentally when boxes are replaced with bricks, with the trivial construction becoming best possible. In light of this, we also consider the special case of Question \ref{qu:kpiercingbox} when all the boxes are assumed to be bricks. We obtain a similar result, even under this additional restriction.

\begin{theorem}\label{thm:kpiercingbrick}
Let $k \ge 2$ and $d\ge 1$ be integers. For $n$ large enough there exists a partition $\{B^1,\dots,B^m\}$ of $[n]^d$ into proper bricks having the $k$-piercing property with $m\le 3.92^{d}k$.
\end{theorem}

Both proofs involve building an intermediate partition coming from a low-dimensional example and then solving a smaller instance of the same problem within each part. It seems almost certain that better examples exist, and in fact it is not out of the question that $m=(2+o(1))^d$ for every fixed $k$, in both regimes.

For the lower bounds, there is a simple inclusion-exclusion argument which shows $m \ge d2^dk,$ but this only applies for bricks. With boxes, lower bounds are difficult to obtain, as neither the argument mentioned above nor the one used to prove Theorem \ref{thm:alon} seem to extend to this problem. In fact, we fail to obtain any lower bound of the form $(1+ \eps)^dk$ for any $\eps>0. $ Such a bound almost certainly holds, and this presents a very interesting open problem. 

In this setting even the $2$-dimensional case is not easy to resolve. The upper bound of $m\le 4k-4$ follows from the left image in Figure~\ref{fig:firstkp} and is easily seen to be tight in the case of bricks. With the aim of showing that this is best possible even for boxes, we introduce a graph theory question of an extremal flavour and solve it asymptotically. This gives the following result.

\begin{proposition}\label{prop:2dim}
Let $\{B^1,\dots,B^m\}$ be a minimal partition of $[n]^2$ into proper sub-boxes satisfying the $k$-piercing property. Then, assuming $n \ge 2k-2$ we have $m=(4+o_k(1))k$. 
\end{proposition}

This short paper is organized as follows. In Section~\ref{sec:setup} we give some set-up and preliminary observations. In Section~\ref{sec:proofs} we prove Theorem~\ref{thm:main} and Theorem~\ref{prop:bricks}. In Section~\ref{sec:piercing} we consider the $k$-piercing problem and present our results, including Theorem~\ref{thm:kpiercingbox}, Theorem \ref{thm:kpiercingbrick} and Proposition~\ref{prop:2dim}. A selection of open questions are given in Section~\ref{sec:concl}.

Before beginning with the set-up for our investigations, we draw attention to other variants of the problem which have been considered in the literature, including geometrical results concerning the minimal partitions obtained in Theorem~\ref{thm:alon}~\cite{Krzysztof2} and extensions of these ideas in the context of cube tiling~\cite{Krzysztof}.

\section{Set-up and previous results}\label{sec:setup}

We begin this section by giving the proof of Alon, Bohman, Holzman and Kleitman of Theorem~\ref{thm:alon}, as presented in~\cite{saks}.
\begin{proof}[Proof of Theorem~\ref{thm:alon}]
Let $A=A_1\times A_2\times\ldots\times A_d$ be a $d$-dimensional discrete box and let $\{B^1,B^2,\ldots,B^m\}$ be a partition of $A$ into proper sub-boxes, where $B^j=B^j_1\times B^j_2\times\ldots\times B^j_d$ for all $j$.  Select sets $R_i$, $i\in\{1,2,\ldots,d\}$, independently, uniformly at random amongst all odd-sized subsets of $A_i$, and let $R:=R_1\times R_2\times\ldots\times R_d$. 

For $j\in\{1,2,\ldots,m\}$, let $X_j$ be the indicator function of the event that $|B^j\cap R|$ is odd, and set $X=\sum_{j=1}^m X_j$. Then we have that the expectation of $X_j$ satisfies
$$\mathbb{E}(X_j)=\mathbb{P}\left(|B^j\cap R|\text{ is odd}\right) = \prod_{i=1}^d\mathbb{P}\left(|B^j_i\cap R_i|\text{ is odd}\right)=2^{-d},$$where we have used the observation that half of the odd cardinality subsets of $A_i$ intersect $B^j_i$ in an odd number of elements. By linearity of expectation we have $\mathbb{E}(X)=m2^{-d}$. Note also that
$$X\equiv \sum_j X_j \equiv \sum_j |B^j\cap R|\equiv |R|\equiv 1 \text{ mod }2.$$ Hence $X\geq 1$ with probability $1$, implying that $\mathbb{E}(X)\geq 1$ and so $m\geq 2^d$ as claimed.
\end{proof}

Let $\fo(n,d)$ denote the number of odd proper sub-boxes required to partition the box $[n]^d$. Note that it is easily seen from Theorem~\ref{thm:alon} that whenever $n\geq 2$ is even we have $\fo(n,d) = 2^d$. Hence we will always assume that the first argument of $\fo$ is odd. Using this notation, Theorem~\ref{thm:main} simply states that if $d\geq 3$ is divisible by $3$ then $\fo(5,d)\leq 25^{d/3}$.

Note first that if $m\geq n$ are odd integers, and $\mathcal{B}$ is a partition of $[n]^d$ into odd proper sub-boxes, then one can obtain a partition of $[m]^d$ into $|\mathcal{B}|$ odd proper sub-boxes by identifying the element $\{n\}$ with the interval $\{n,n+1,\ldots,m\}$. Hence if $2<n\leq m$ are odd integers and $d\geq 1$ then 
\begin{equation}\label{eq:monotone}
\fo(n,d)\geq \fo(m,d).
\end{equation}

Note that if $\mathcal{B}_1$ and $\mathcal{B}_2$ are partitions of $[n]^d_1$ and $[n]^d_2$ respectively into odd boxes, then $\mathcal{B}_1\times \mathcal{B}_2$ is a partition of $[n]^{d_1+d_2}$ into $|\mathcal{B}_1|\cdot |\mathcal{B}_2|$ odd boxes. Hence the function $\fo$ satisfies 
\begin{equation}\label{eq:submult}
\fo(n,d_1+d_2)\leq\fo(n,d_1)\cdot \fo(n,d_2)
\end{equation} for all $n\geq 2$ and $d_1,d_2\geq 1$. Since by Theorem~\ref{thm:alon} we have that $\fo(n,d)\geq 2^d$ for all $n,d$, Fekete's lemma~\cite{fekete} can be applied. It follows that for every $n\geq 2$, there exists a nonnegative constant $\alpha_n$ depending only on $n$, such that $\fo(n,d)=\left(\alpha_n+o_d(1)\right)^d$, where $o_d(1)\rightarrow 0$ as $d\rightarrow \infty$.

By inequality~\eqref{eq:monotone} the sequence $(\alpha_n)_{n\in\mathbb{N}}$ is monotone decreasing. An interesting open question is whether the limit of the sequence on the odd integers is equal to two or not -- see Section~\ref{sec:concl} for more details.

Note that these considerations apply equally to the $k$-piercing problem, showing that for fixed $k$ the minimum number of boxes in a partition with the $k$-piercing property is at least $(\beta_{k,n}+o_{d}(1))^d$ for some monotone decreasing sequence $(\beta_{k,n})_{n\in \mathbb{N}}$. Letting $\beta_k=\lim_{n\to \infty}\beta_{k,n}$, Theorem~\ref{thm:kpiercingbox} shows that $\beta_k\le 15^{1/2}$ for all $k$. Similarly, one can define $\gamma_k$ for the case of bricks, in which case Theorem~\ref{thm:kpiercingbrick} implies $\gamma_k \le 3.92.$

Let us denote by $p_{\text{box}}(n,d,k)$ the answer to Question~\ref{qu:kpiercingbox} 
and by $p_{\text{brick}}(n,d,k)$ the answer to the same question, but restricted to bricks. Let $p_{\text{box}}(d,k)=\lim_{n\rightarrow\infty}p_{\text{box}}(n,d,k)$ and $p_{\text{brick}}(d,k)=\lim_{n\rightarrow\infty}p_{\text{brick}}(n,d,k)$, which both exist by the above observations. 
As any brick is a box, we know that $p_{\text{box}}(d,k) \le p_{\text{brick}}(d,k).$  Note that with the above definitions $p_{\text{brick}}(d,k)=(\beta_k+o_d(1))^d$ and $p_{\text{box}}(d,k)=(\gamma_k+o_d(1))^d.$ 

The case of $k=2$ is resolved completely by Theorem~\ref{thm:alon} as there is a trivial partition into $2^d$ bricks, by splitting the original box into two along each dimension, implying $p_{\text{brick}}(d,2) \le 2^d $. On the other hand, a partition being $2$-piercing is equivalent to it consisting only of proper boxes, so Theorem~\ref{thm:alon} implies that $2^d \le p_{\text{box}}(d,2).$ In particular, this implies a very surprising result that for $k=2$ the answer is the same for boxes and bricks: $p_{\text{box}}(d,2) = p_{\text{brick}}(d,2) =2^d.$

\section{Partitioning into odd boxes}\label{sec:proofs}

We start with proving the upper bound, given by Theorem~\ref{thm:main}. 
\begin{proof}[Proof of Theorem~\ref{thm:main}]
Note that by inequality~\eqref{eq:submult}, it suffices to show that $\fo(5,3)\leq 25$. That is, we seek a partition of $[5]^3$ into $25$ proper odd boxes. This partition can be seen on Figure~\ref{fig:25oddbox}. The list of the coordinates of the $25$ boxes can be found in the appendix.

This solution was found by phrasing the problem as an integer program, with one (Boolean) variable for every possible odd sub-box, and one constraint per coordinate saying that the sum of variables that correspond to boxes which contain this point is one. We then used Gurobi~\cite{gurobi}, a commercially available solver, to find the counterexample.
\end{proof}

We now turn to lower bounds, starting with the easy observation that for each fixed $n$ we have $\alpha_n>2$.

\begin{proposition}\label{prop:oddlower}
Let $n>2$ be odd, and $d\geq 1$. Then we have the lower bound $$\fo(n,d)\geq \left(2+\frac{1}{2^{n-2}-1}\right)^d.$$
\end{proposition}
\begin{proof}
The proof of Proposition~\ref{prop:oddlower} is a trivial modification of the proof of Alon, Bohman, Holzman and Kleitman of Theorem~\ref{thm:alon}. We simply take the sets $R_i$ to be uniformly chosen at random amongst \emph{proper}, odd-sized subsets of $[n]$. That is, $R_i$ is a uniformly random element of the set $\{S\subset A : S\neq A \text{ and }|S| \text{ is odd}\}$. Define $X_j, X$ and $R$ as in the proof of Theorem~\ref{thm:alon} and note that
$$\mathbb{E}(X_j)=\mathbb{P}\left(|B^j\cap R|\text{ is odd}\right)=\left(\frac{2^{n-2}-1}{2^{n-1}-1}\right)^d.$$ As before we have that $X\geq 1$ with probability $1$, hence $\mathbb{E}(X)=m\mathbb{E}(X_j)\geq 1$. After rearranging, this gives the required result. 
\end{proof}

Note that Proposition~\ref{prop:oddlower} simply says that $\alpha_n\geq 2+\frac{1}{2^{n-2}-1}$ for all odd $n$, but this sequence of lower bounds on the $\alpha_n$-s converges to two. 

We will now consider the case where the members of our partition are proper, odd bricks. The idea behind the proof of Theorem~\ref{prop:bricks} is to remove the `top' and `bottom' layers of a partition and prove that the number of remaining bricks has to be large, since their projection onto the first $d-1$ layers forms a partition of a $d-1$-dimensional odd box. While this is not quite true, this proof method can be made to work by considering a stronger induction hypothesis.

\begin{proof}[Proof of Proposition~\ref{prop:bricks}]
Let $n\geq 2$ be odd, $d\geq 1$ arbitrary integer. We will prove the stronger claim that if $\mathcal{B}=\{B^1,B^2,\ldots,B^m\}$ is a set of odd proper bricks that cover every element of $[n]^d$ an odd number of times, then $m\geq 3^d$. The proof goes by induction on $d$.

Let $n,d,\mathcal{B}$ be given. For any brick $B \in \mathcal{B}$ let $B=B_1 \times \cdots \times B_d$ where $B_i$ are odd length intervals. Let $\mathcal{C},\mathcal{D}\subset\mathcal{B}$ be defined as
$$\mathcal{C}=\left\{B^i : B^i\cap\left( \underbrace{[n] \times [n] \times \ldots \times [n]}_{d-1}\times \{1\}\right)\neq \emptyset\right\},$$
and
$$\mathcal{D}=\left\{B^i : B^i\cap\left( \underbrace{[n] \times [n] \times \ldots \times [n]}_{d-1}\times \{n\}\right)\neq \emptyset\right\}.$$
Note that $\mathcal{C}\cap \mathcal{D}=\emptyset$, as all $B^i$-s are proper bricks. Moreover, as elements of $\mathcal{C}$ cover every point of $[n]^{d-1}\times\{1\}$ an odd number of times, by induction we have $|\mathcal{C}|\geq 3^{d-1}$, and similarly $|\mathcal{D}|\geq 3^{d-1}$. Remains to show that $|\mathcal{B}\setminus (\mathcal{C}\cup\mathcal{D})|\geq 3^{d-1}$.

For every point $(i_1,i_2,\ldots,i_d)\in [n]^d$ and any family of bricks $\mathcal{E}$, denote by $x_{i_1,i_2,\ldots,i_d}(\mathcal{E})$ the number of bricks in $\mathcal{E}$ that contain $\{i_1\}\times\{i_2\}\times\ldots\times\{i_d\}$, and note that by assumptions $x_{i_1,i_2,\ldots,i_d}(\mathcal{B})$ is odd for all choices of the $i_j$-s. 

For all $(i_1,i_2,\ldots,i_{d-1})\in[n]^{d-1}$ define the quantity
$$y_{i_1,i_2,\ldots,i_{d-1}}=\sum_{j=1}^n x_{i_1,i_2,\ldots,i_{d-1},j}(\mathcal{B}\setminus (\mathcal{C}\cup \mathcal{D})),$$
and note that $y_{i_1,i_2,\ldots,i_{d-1}}$ is odd for all choices of  $i_1,\ldots,i_{d-1}$. Indeed, as $\mathcal{C}\cap \mathcal{D}=\emptyset$ 
\begin{align*}
y_{i_1,i_2,\ldots,i_{d-1}} 
& =\sum_{j=1}^n x_{i_1,i_2,\ldots,i_{d-1},j}(\mathcal{B})-\sum_{j=1}^n x_{i_1,i_2,\ldots,i_{d-1},j}(\mathcal{C})-\sum_{j=1}^n x_{i_1,i_2,\ldots,i_{d-1},j}(\mathcal{D})\\
& =\sum_{j=1}^n x_{i_1,i_2,\ldots,i_{d-1},j}(\mathcal{B})-\sum_{C \in \mathcal{C}}\sum_{j=1}^n \mathbbm{1} ((i_1,i_2,\ldots,i_{d-1},j) \in \mathcal{C})-\sum_{D \in \mathcal{D}}\sum_{j=1}^n \mathbbm{1} ((i_1,i_2,\ldots,i_{d-1},j) \in \mathcal{D}) \\
& = \sum_{j=1}^n x_{i_1,i_2,\ldots,i_{d-1},j}(\mathcal{B})-\sum_{C \in \mathcal{C}}|C_d|-\sum_{D \in \mathcal{D}}|D_d|,
\end{align*}
where $\mathbbm{1}(\cdot)$ denotes the indicator function of an event.

Now as $C_d,D_d$ are odd size intervals each term in the above sums is odd, so the total residue mod $2$ is $n-|\mathcal{C}|-|\mathcal{D}|,$ which is odd.

Consider the projection of the bricks in $\mathcal{B}\setminus(\mathcal{C}\cup\mathcal{D})$ onto the first $d-1$ coordinates and note that it induces an odd cover of a $d-1$ dimensional odd cube, as follows. For any brick $B\in \mathcal{B}\setminus(\mathcal{C}\cup\mathcal{D})$ define $\pi(B):= B_1\times B_2\times \ldots \times B_{d-1}$ to be the projection of the box $B$ onto the first $d-1$ coordinates. For all $(i_1,i_2,\ldots,i_{d-1})\in [n]^{d-1}$ define the quantity
$$z_{i_1,i_2,\ldots,i_{d-1}}=\sum_{B\in\mathcal{B}\setminus(\mathcal{C}\cup\mathcal{D})}\mathbbm{1}\left((i_1, i_2,\ldots,i_{d-1})\in \pi(B)\right).$$ 
Observe that 
$$z_{i_1,i_2,\ldots,i_{d-1}}\equiv y_{i_1,i_2,\ldots,i_{d-1}} \text{ mod }2$$
for all choices of coordinates, and hence all $z_{i_1,i_2,\ldots,i_{d-1}}$-s are odd. Since the set of bricks 
$$\{\pi(B): B\in\mathcal{B}\setminus(\mathcal{C}\cup\mathcal{D})\}$$
form a cover of $[n]^{d-1}$ with each point covered $z_{i_1,i_2,\ldots,i_{d-1}}$ times, it follows by induction that $|\mathcal{B}\setminus(\mathcal{C}\cup\mathcal{D})|\geq 3^{d-1}$ and the proof is complete.
\end{proof}

\section{Piercing}\label{sec:piercing}
In this section we will consider piercing problems related to Question~\ref{qu:kpiercingbox}. We start by giving some simple bounds, derived by generalising the arguments used for $k=2,$ which illustrate various difficulties that arise. In the following subsections we give various improvements to these bounds. 

In the case of bricks, observe that a single brick of the partition that does not contain a corner vertex can be incident to only one edge of the original cube, as otherwise it would not be proper and thus fail the $k$-piercing property (even for $k=2$). Also, for each edge there needs to be at least $k$ boxes which are incident to it. Combining these two observations we deduce that there needs to be at least $d2^{d-1}(k-2)$ different non-corner boxes, as there are $d2^{d-1}$ edges. Including the additional $2^d$ corner boxes this implies that there are at least $d2^{d-1}(k-2)+2^d$ different boxes. 
On the other hand, generalising the partition used for $k=2$, splitting the original cube into $k$ parts along each dimension obtains a $k$-piercing partition into $k^d$ bricks. So we have shown the following two easy bounds:
\begin{equation} \label{eq:brick-trivial}
d2^{d-1}(k-2)+2^d \le p_{\text{brick}}(d,k) \le k^d.    
\end{equation}

In the case of boxes, the lower bound no longer applies, as almost all the bricks counted as different above might become parts of a single box. The same kind of argument only gives $p_{\text{box}}(d,k) \ge d(k-1)+1$ by fixing a corner and counting all the boxes incident to an edge containing this corner, which need to be different. Furthermore, it is not clear how to exploit the $k$-piercing property in the argument used in Theorem~\ref{thm:alon} for $k>2$. However Theorem~\ref{thm:alon} is directly applicable in the case $k=2$, which gives a lower bound of $2^d$ which then holds for all $k\ge 2$. From the other direction, it is also not clear how one could exploit the freedom afforded by using boxes instead of bricks when trying to find a partition, and in fact when $k=2$ this turns out not to be possible. We can, however, reuse the bound for bricks to obtain the following simple bounds:
\begin{equation}\label{eq:box-trivial}
\max (k(d-1)+1,2^d) \le p_{\text{box}}(d,k) \le k^d.
\end{equation}

Note that the lower bound for $p_{\text{box}}$ highlights a disconnect between our methods for dealing with the two most extreme regimes: firstly the case of $k$ fixed and $d\to \infty$ in which the lower bound is $2^d$, and secondly the case of $d$ fixed and $k\to\infty$ in which the bound of $k(d-1)+1$ is relevant. We shall give our results in terms of both $k$ and $d$ so that they apply generally, and indeed the upper bounds we shall describe are the best we know across all regimes. Our lower bound efforts, however, are most relevant for the latter scenario (when $d$ is small compared to $k$). 

In the following subsections we will describe our various improvements to the above bounds. In the first subsection we will discuss upper bounds on $p_{\text{brick}}(d,k)$ and $p_{\text{box}}(d,k)$ and in the second subsection we discuss lower bounds.

\subsection{Upper bounds for the $k$-piercing problem}
In this section we will present the proof of Theorems~\ref{thm:kpiercingbox} and \ref{thm:kpiercingbrick}, giving a major improvement over the upper bound in \eqref{eq:brick-trivial} and \eqref{eq:box-trivial}. We begin by presenting a simple partition into at most $4^dk$ bricks that satisfies the $k$-piercing property. This construction is so simple and natural that one might imagine that it could be best possible. This is not the case, however, and we will go on to present two different approaches for obtaining improvements in the base of the exponent, one of which is specific for boxes and gives a slightly better bound.

We define $f_d(a_1,\dots,a_d)$ to be the minimum size of a partition of $[n]^d$ into boxes so that every line in dimension $i$ hits at least $a_i$ boxes, (we refer to this as the $(a_1,\dots,a_d)$-piercing condition). In the first two dimensions, we split $[n]^d$ into 4 quadrants. In the top left and bottom right quadrants we place a construction satisfying the $(1,k-1,k,\dots,k)$-piercing condition. In the bottom left and top right quadrants we place a construction satisfying the $(k-1,1,k,\dots,k)$-piercing condition. This is shown in Figure~\ref{fig:firstkp}. This gives a construction satisfying the $k$-piercing condition, and observing that $f_d(1,k,\dots,k)\le f_{d-1}(k-1,k,\dots,k) \le f_{d-1}(k,k,\dots,k)$ gives the following bound for $d\ge 2$:
\[f_d(k,\dots,k)\le 4f_{d-1}(k,\dots,k).\]
Combining this with the fact that $f_1(k)=k$ we find that $f_d(k,\dots,k)\le 4^d k$.

\begin{figure}
\caption{On the left we see a $k$-piercing configuration in two dimensions with $4(k-1)$ bricks. On the right, we use this idea to give a $k$-piercing construction with $k4^d$ boxes. In the first two dimensions we divide the cube into quadrants and then place optimal constructions in each quadrant satisfying the piercing conditions shown.
}
\begin{center}
\begin{tikzpicture}
\draw [-] (0,0) -- (6,0);
\draw [-] (0,0) -- (0,6);
\draw [-] (0,6) -- (6,6);
\draw [-] (6,0) -- (6,6);

\draw [-] (8,0) -- (14,0);
\draw [-] (8,0) -- (8,6);
\draw [-] (8,6) -- (14,6);
\draw [-] (14,0) -- (14,6);

\draw [very thin] (9.5,1.5) node[] {\scalebox{0.8}{$\Spvek{1;k-1;k;\vdots;k}$}};
\draw [very thin] (9.5,4.5) node[] {\scalebox{0.8}{$\Spvek{k-1;1;k;\vdots;k}$}};
\draw [very thin] (12.5,4.5) node[] {\scalebox{0.8}{$\Spvek{1;k-1;k;\vdots;k}$}};
\draw [very thin] (12.5,1.5) node[] {\scalebox{0.8}{$\Spvek{k-1;1;k;\vdots;k}$}};

\draw [very thin] (3,7.5) node[] {a) \,\, $p_{\text{brick}}(2,k)\le 4(k-1)$};
\draw [very thin] (11,7.5) node[] {b) \,\, $p_{\text{brick}}(d,k)\le 4p_{\text{brick}}(d-1,k)$};

\draw [-] (11,0) -- (11,6);
\draw [-] (8,3) -- (14,3);
\draw [-] (3,0) -- (3,6);
\draw [-] (0,3) -- (6,3);
\draw [-] (0,0.5) -- (3,0.5);
\draw [-] (0,1) -- (3,1);
\draw [-] (0,2.5) -- (3,2.5);

\draw [very thin] (1.5,1.85) node[] {\scalebox{1.5}{\vdots}};

\draw [-] (3.5,0) -- (3.5,3);
\draw [-] (4,0) -- (4,3);
\draw [-] (5.5,0) -- (5.5,3);

\draw [very thin] (4.85,1.5) node[] {\scalebox{1.5}{\dots}};

\draw [-] (0.5,3) -- (0.5,6);
\draw [-] (1,3) -- (1,6);
\draw [-] (2.5,3) -- (2.5,6);

\draw [very thin] (1.85,4.5) node[] {\scalebox{1.5}{\dots}};

\draw [-] (3,3.5) -- (6,3.5);
\draw [-] (3,4) -- (6,4);
\draw [-] (3,5.5) -- (6,5.5);

\draw [very thin] (4.5,4.85) node[] {\scalebox{1.5}{\vdots}};

\draw [decorate,decoration={brace,amplitude=10pt,raise=4pt},yshift=0pt]
(0,0) -- (0,3) node [black,midway,xshift=-1cm] {\footnotesize
$k-1$};

\draw [decorate,decoration={brace,amplitude=10pt,raise=4pt},yshift=0pt]
(6,0) -- (3,0) node [black,midway,yshift=-0.8cm] {\footnotesize
$k-1$};

\draw [decorate,decoration={brace,amplitude=10pt,raise=4pt},yshift=0pt]
(6,6) -- (6,3) node [black,midway,xshift=1cm] {\footnotesize
$k-1$};

\draw [decorate,decoration={brace,amplitude=10pt,raise=4pt},yshift=0pt]
(0,6) -- (3,6) node [black,midway,yshift=0.8cm] {\footnotesize
$k-1$};

\end{tikzpicture}
\end{center}
\label{fig:firstkp}
\end{figure}
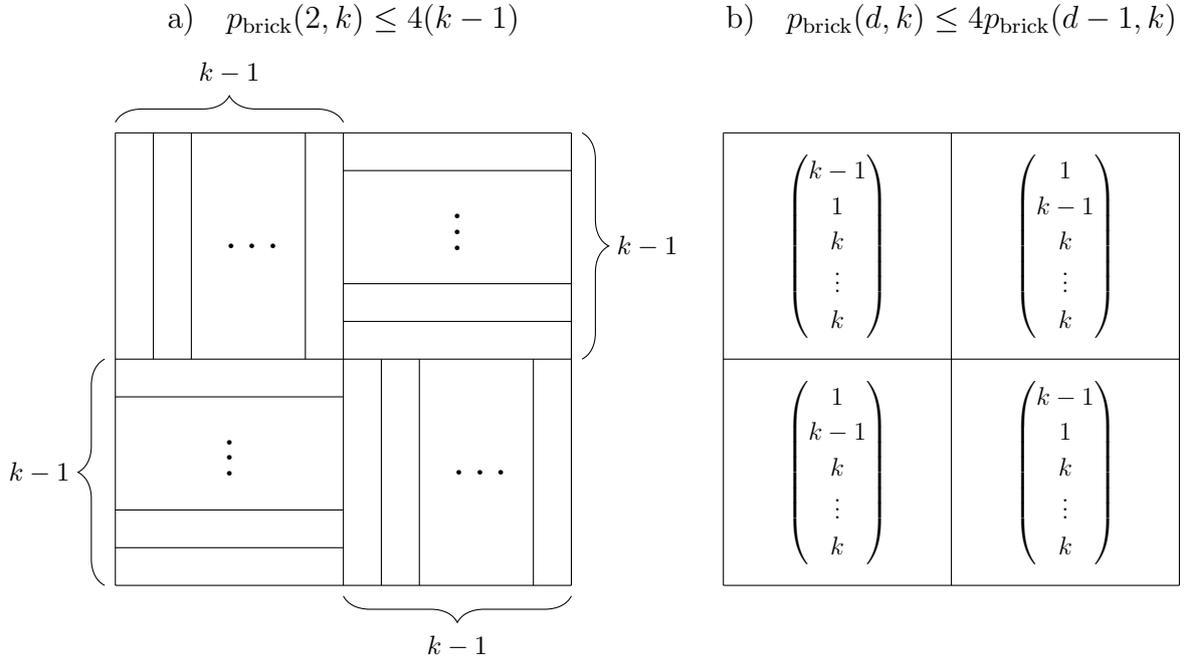

In particular this shows:
\begin{equation}\label{eq:4tothed}
p_{\text{box}}(d,k) \le p_{\text{brick}}(d,k) \le 4^dk     
\end{equation}
So, in the notation introduced in Section~\ref{sec:setup}, we have $\gamma_k \le \beta_k \le 4.$ 

One may wonder if these bounds are tight, and the construction describes above is essentially best possible (at least in the case of bricks). We will now show that this is not the case, and give two different approaches for improving the base of the exponent further. In both following subsections we will reuse the general idea of splitting the cube along a couple of dimensions. In the following subsection we work with bricks and prove Theorem~\ref{thm:kpiercingbrick} and in the subsequent subsection we exploit a simple observation which holds for boxes but not for bricks to get an even better bound.

\subsubsection{Bricks}
In some sense a more surprising part of the result \eqref{eq:4tothed} is the fact that for a fixed dimension $d$ both $p_{\text{box}}(d,k)$ and $p_{\text{brick}}(d,k)$ are linear in $k$, but using the sub-multiplicative inequalities such as \eqref{eq:submult} can never give results linear in $k.$ The idea of finding a small example and then using these inequalities as was done in the previous section for $f_{odd}$ can only ever give something interesting when $k$ is rather small. However, the idea behind the argument giving \eqref{eq:4tothed} is to use small examples in a different manner. The following observation gives a more general view of this idea.

Given a partition of $[n]^d$ into bricks $A_1,\ldots, A_m$ such that we can assign to each $A_i$ a $d$-tuple $(a_{i,1}\ldots,a_{i,d})$ of positive integers such that for any line in $j$-th dimension the sum of $a_{i,j},$ with $i$ ranging over the bricks crossed by this line, is at least $k.$ Whenever we have such a partition we obtain that $f_d(k,\ldots,k) \le \sum_{i=1}^m f_d(a_{i,1},\ldots,a_{i,d})$ as we can solve the corresponding subproblem within each brick of the partition. We will call such a partition \emph{intermediate}.

The natural goal is to find small examples of intermediate partitions. For example, given a $k$-piercing example for small $d,$ if we can group several bricks into sets $A_i$ to obtain an intermediate partition then we obtain an upper bound on $f_d(k,\ldots,k).$ For instance, in the proof of \eqref{eq:4tothed}, we used the example on the left of Figure~\ref{fig:firstkp} which gives a natural grouping into $4$ bricks, yielding the intermediate example on the right of this figure.

The following lemma gives a way of obtaining, from an intermediate partition in $d$ dimensions, a new intermediate partition in $d+1$ dimensions in a slightly better way than the trivial approach of stacking two copies on top of one another. 
\begin{lemma}\label{lem:proper-subbox}
Let $A_1,\ldots,A_m$ be an intermediate partition of $[n]^d.$ Let $X$ and $Y$ be corners of the cube such that the largest proper sub-brick containing $X$ covers w.l.o.g. $A_1, \ldots, A_s$ and let $A_r$ be the brick containing corner $Y.$ Then
\begin{align*}
    f_{d+1}(k,\ldots,k)\le & \sum_{i=1}^s f_{d+1}(a_{i,1}\ldots,a_{i,d},1)+\sum_{i=s+1}^m f_{d+1}(a_{i,1}\ldots,a_{i,d},k-1)+\\ 
    & \sum_{i=1,i \neq r}^mf_{d+1}(a_{i,1}\ldots,a_{i,d},1)+f_{d+1}(a_{r,1}\ldots,a_{r,d},k-1).
\end{align*}

\end{lemma}
\begin{proof}
We split the cube in two parts along the $d+1$-st dimension. We use the given partition for both parts, but with the top part rotated in such a way that $Y$ corresponds to $X.$ We then rescale the top partition in such a way that $A_r$ covers all of $A_1, \ldots, A_s$ in the original partition (note that his may require a minor increase in the $n$ we use). We add $k-1$ for the last dimension of $A_r$ in the top part and all the bricks in the lower part except $A_1, \ldots, A_s,$ we add $1$ for the remaining bricks. This new partition is a new intermediate partition in $d+1$ dimensions, as along first $d$ dimensions all the lines satisfy the condition because we started with an intermediate partition, and along the $d+1$'st, if it passes through any of $A_1, \ldots A_s$ of the lower part it will pass through $A_r$ of the upper part so the sum will be at least $1+k-1$ and, otherwise it will pass through some $A_i,$ $i \ge s+1$ in the lower part and something in the upper part again giving sum at least $k-1+1.$ The inequality now follows from the above observation. 
\end{proof}

We now apply this lemma to the $5$-part intermediate partition, derived from the one given in Figure~\ref{fig:firstkp}, and given in Figure~\ref{fig:secondkp}. We obtain the $3$-dimensional intermediate partition shown in Figure~\ref{fig:thirdkp}.
\begin{figure}
\caption{The intermediate partition in 2 dimensions, to which we apply Lemma~\ref{lem:proper-subbox}. $X$ is denoted by red circle, $Y$ by a blue circle, the parts $A_1,\ldots,A_s$ are shaded red and $A_r$ is shaded blue.}
\begin{center}
\begin{tikzpicture}
\draw [-] (0,0) -- (6.25,0); 
\draw [-] (0,0) -- (0,5); 
\draw [-] (0,5) -- (6.25,5); 
\draw [-] (4,0) -- (4,2.5);
\draw [-] (6.25,0) -- (6.25,5);

\draw [-] (2.5,0) -- (2.5,5);
\draw [-] (0,2.5) -- (6.25,2.5);
\draw[thick] (0,0) circle(0.3)[red];
\draw[thick] (6.25,0) circle(0.3)[blue];
\fill[red,opacity=0.2] (0,0) rectangle (4,2.5);
\fill[blue,opacity=0.2] (4,0) rectangle (6.25,2.5);
\draw [very thin] (1.25,1.25) node[] {{$\Spvek{1;k-1}$}};
\draw [very thin] (1.25,3.75) node[] {{$\Spvek{k-1;1}$}};
\draw [very thin] (3.25,1.25) node[] {\footnotesize{$\Spvek{k-2;1}$}};
\draw [very thin] (4.375,3.75) node[] {{$\Spvek{1;k-1}$}};
\draw [very thin] (5.125,1.25) node[] {\footnotesize{$\Spvek{1;1}$}};
\end{tikzpicture}
\end{center}
\label{fig:secondkp}
\end{figure}
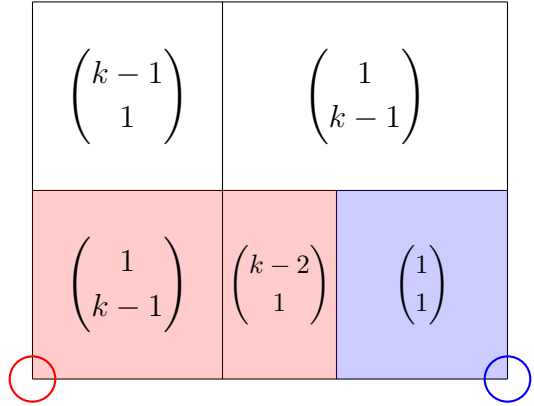

\begin{figure}
\caption{The intermediate partition in 3 dimensions, provided by the above lemma.}
\begin{center}
\begin{tikzpicture}
\draw [-] (0,0) -- (6.25,0); 
\draw [-] (0,0) -- (0,5); 
\draw [-] (0,5) -- (6.25,5); 
\draw [-] (3.75,0) -- (3.75,2.5);
\draw [-] (6.25,0) -- (6.25,5);

\draw[thick] (0,0) circle(0.3)[red];
\draw [-] (2.5,0) -- (2.5,5);
\draw [-] (0,2.5) -- (6.25,2.5);
\fill[red,opacity=0.2] (0,0) rectangle (3.75,2.5);
\draw [very thin] (3.125,-1) node[] {Bottom Layer};
\draw [very thin] (1.25,1.25) node[] {\footnotesize{$\Spvek{1;k-1;1}$}};
\draw [very thin] (1.25,3.75) node[] {\footnotesize{$\Spvek{k-1;1;k-1}$}};
\draw [very thin] (3.125,1.25) node[] {\tiny{$\Spvek{k-2;1;1}$}};
\draw [very thin] (4.375,3.75) node[] {\footnotesize{$\Spvek{1;k-1;k-1}$}};
\draw [very thin] (5,1.25) node[] {\tiny{$\Spvek{1;1;k-1}$}};

\draw [-] (8,0) -- (14.25,0); 
\draw [-] (8,0) -- (8,5); 
\draw [-] (8,5) -- (14.25,5); 
\draw [-] (11.75,0) -- (11.75,2.5);
\draw [-] (14.25,0) -- (14.25,5);

\draw [-] (13,0) -- (13,5);
\draw [-] (8,2.5) -- (14.25,2.5);
\draw[thick] (8,0) circle(0.3)[blue];
\fill[blue,opacity=0.2] (8,0) rectangle (11.75,2.5);
\draw [very thin] (11.125,-1) node[] {Top Layer};
\draw [very thin] (9.875,1.25) node[] {\footnotesize{$\Spvek{1;1;k-1}$}};
\draw [very thin] (10.5,3.75) node[] {\footnotesize{$\Spvek{1;k-1;1}$}};
\draw [very thin] (12.375,1.25) node[] {\tiny{$\Spvek{k-2;1;1}$}};
\draw [very thin] (13.625,3.75) node[] {\tiny{$\Spvek{k-1;1;1}$}};
\draw [very thin] (13.625,1.25) node[] {\tiny{$\Spvek{1;k-1;1}$}};
\end{tikzpicture}
\end{center}
\label{fig:thirdkp}
\end{figure}
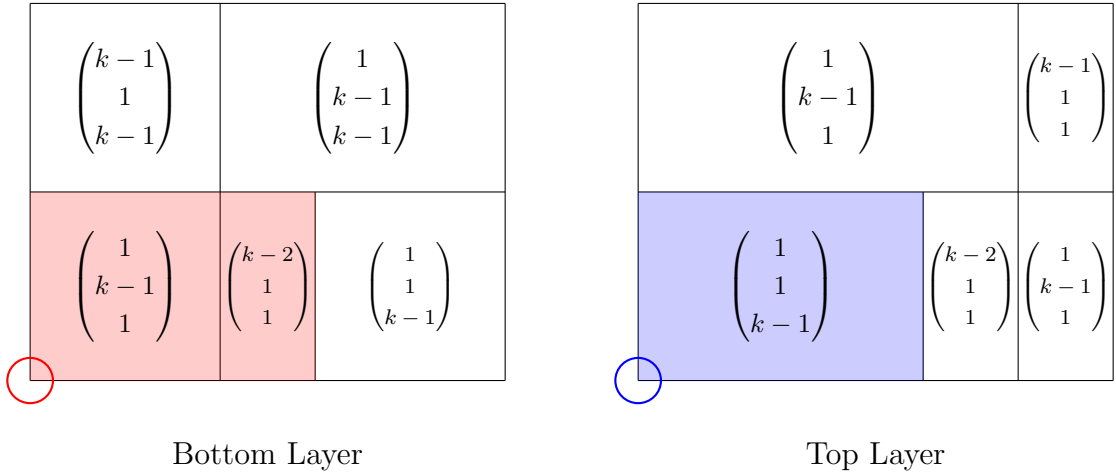

In particular, this implies:
$$f_d(k,\ldots,k) \le 2f_d(1,k-1,k-1,k,\ldots,k)+6f_d(1,1,k-1,k,\ldots,k)+2f_d(1,1,k-2,k,\ldots,k)$$

Unfortunately, this bound still only implies $f_d(k,\ldots,k) \le (4+o_d(1))^dk,$ but modifying this partition slightly, we consider Figure~\ref{fig:fourthkp} and apply the lemma once again. This does achieve an improvement in the base of the exponential term.

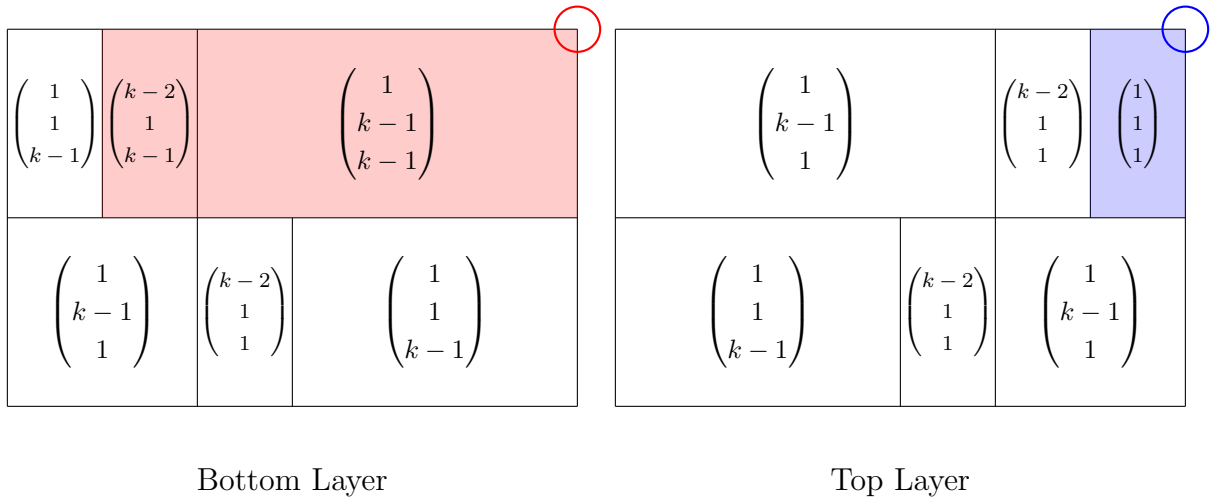
\begin{figure}
\caption{The intermediate partition in 3 dimensions, to which we apply Lemma~\ref{lem:proper-subbox}. $X$ is denoted by red circle, $Y$ by a blue circle, the parts $A_1,\ldots,A_s$ are shaded red and $A_r$ is shaded blue.}
\begin{center}
\begin{tikzpicture}
\draw [-] (0,0) -- (7.5,0); 
\draw [-] (0,0) -- (0,5); 
\draw [-] (0,5) -- (7.5,5); 
\draw [-] (3.75,0) -- (3.75,2.5);
\draw [-] (7.5,0) -- (7.5,5);
\draw [-] (1.25,2.5) -- (1.25,5);

\draw[thick] (7.5,5) circle(0.3)[red];
\draw [-] (2.5,0) -- (2.5,5);
\draw [-] (0,2.5) -- (7.5,2.5);
\fill[red,opacity=0.2] (1.25,2.5) rectangle (7.5,5);
\draw [very thin] (3.75,-1) node[] {Bottom Layer};
\draw [very thin] (1.25,1.25) node[] {\footnotesize{$\Spvek{1;k-1;1}$}};
\draw [very thin] (5,3.75) node[] {\footnotesize{$\Spvek{1;k-1;k-1}$}};
\draw [very thin] (3.125,1.25) node[] {\tiny{$\Spvek{k-2;1;1}$}};
\draw [very thin] (1.875,3.75) node[] {\tiny{$\Spvek{k-2;1;k-1}$}};
\draw [very thin] (0.625,3.75) node[] {\tiny{$\Spvek{1;1;k-1}$}};
\draw [very thin] (5.625,1.25) node[] {\footnotesize{$\Spvek{1;1;k-1}$}};

\draw [-] (8,0) -- (15.5,0); 
\draw [-] (8,0) -- (8,5); 
\draw [-] (8,5) -- (15.5,5); 
\draw [-] (11.75,0) -- (11.75,2.5);
\draw [-] (15.5,0) -- (15.5,5);
\draw [-] (14.25,2.5) -- (14.25,5);

\draw [-] (13,0) -- (13,5);
\draw [-] (8,2.5) -- (15.5,2.5);
\draw[thick] (15.5,5) circle(0.3)[blue];
\fill[blue,opacity=0.2] (14.25,2.5) rectangle (15.5,5);
\draw [very thin] (11.75,-1) node[] {Top Layer};
\draw [very thin] (9.875,1.25) node[] {\footnotesize{$\Spvek{1;1;k-1}$}};
\draw [very thin] (10.5,3.75) node[] {\footnotesize{$\Spvek{1;k-1;1}$}};
\draw [very thin] (12.375,1.25) node[] {\tiny{$\Spvek{k-2;1;1}$}};
\draw [very thin] (13.625,3.75) node[] {\tiny{$\Spvek{k-2;1;1}$}};
\draw [very thin] (14.875,3.75) node[] {\tiny{$\Spvek{1;1;1}$}};
\draw [very thin] (14.25,1.25) node[] {\footnotesize{$\Spvek{1;k-1;1}$}};
\end{tikzpicture}
\end{center}
\label{fig:fourthkp}
\end{figure}
In particular, we find:
\begin{align*}f_d(k,\ldots,k) \le & 8f_d(1,1,k-1,k-1,k,\ldots,k)+5f_d(1,1,k-2,k-1,k,\ldots,k)+\\
& 8f_d(1,1,1,k-1,k,\ldots,k)+3f_d(1,1,1,k-2,k,\ldots,k)
\end{align*}

This already suffices to give an example with at most about $3.97^dk$ bricks. However, since the red bricks have large piercing values in all but one dimension, it turns out that a further manual step can be made before applying Lemma~\ref{lem:proper-subbox}. In particular, using the partition given in Figure~\ref{fig:fifthkp} we obtain the following slight improvement:
\begin{align*}f_d(k,\ldots,k) \le & 10f_d(1,1,k-1,k-1,k,\ldots,k)+3f_d(1,1,k-2,k-1,k,\ldots,k)+\\
& 6f_d(1,1,1,k-1,k,\ldots,k)+3f_d(1,1,1,k-2,k,\ldots,k).
\end{align*}
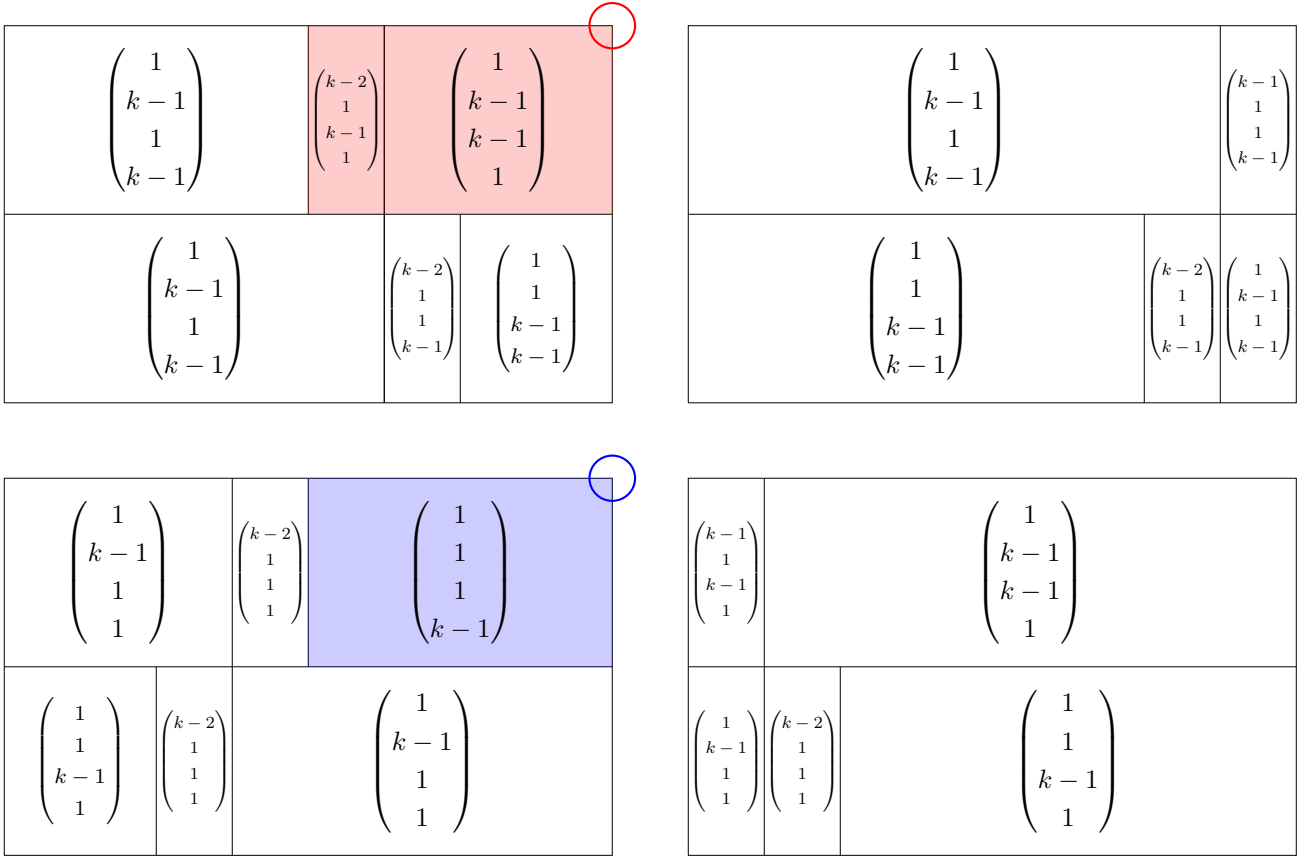
\begin{figure}
\caption{An intermediate partition in 4 dimensions. The third and fourth dimensions move between the rectangles horizontally and vertically respectively.}
\begin{center}
\begin{tikzpicture}
\draw [very thin] (0,0) -- (8,0) -- (8,5) -- (0,5) -- (0,0);
\draw [very thin] (9,0) -- (17,0) -- (17,5) -- (9,5) -- (9,0);
\draw [very thin] (0,6) -- (8,6) -- (8,11) -- (0,11) -- (0,6);
\draw [very thin] (9,6) -- (17,6) -- (17,11) -- (9,11) -- (9,6);
\draw [very thin] (0,2.5) -- (8,2.5);
\draw [very thin] (0,8.5) -- (8,8.5);
\draw [very thin] (9,2.5) -- (17,2.5);
\draw [very thin] (9,8.5) -- (17,8.5);
\draw [very thin] (10,0) -- (10,5);
\draw [very thin] (11,0) -- (11,2.5);
\draw [very thin] (2,0) -- (2,2.5);
\draw [very thin] (3,0) -- (3,5);
\draw [very thin] (4,2.5) -- (4,5);
\draw [very thin] (4,8.5) -- (4,11);
\draw [very thin] (5,6) -- (5,11);
\draw [very thin] (6,6) -- (6,8.5);
\draw [very thin] (5,6) -- (5,11);
\draw [very thin] (15,6) -- (15,8.5);
\draw [very thin] (16,6) -- (16,11);
\draw[thick] (8,5) circle(0.3)[blue];
\draw[thick] (8,11) circle(0.3)[red];
\fill[red,opacity=0.2] (4,8.5) rectangle (8,11);
\fill[blue,opacity=0.2] (4,2.5) rectangle (8,5);
\draw [very thin] (2.5,1.25) node[] {\tiny{\scalebox{0.8}{$\Spvek{k-2;1;1;1}$}}};
\draw [very thin] (1,1.25) node[] {\tiny{\scalebox{1}{$\Spvek{1;1;k-1;1}$}}};
\draw [very thin] (5.5,1.25) node[] {\footnotesize{\scalebox{1}{$\Spvek{1;k-1;1;1}$}}};
\draw [very thin] (1.5,3.75) node[] {\footnotesize{\scalebox{1}{$\Spvek{1;k-1;1;1}$}}};
\draw [very thin] (3.5,3.75) node[] {\tiny{\scalebox{0.8}{$\Spvek{k-2;1;1;1}$}}};
\draw [very thin] (6,3.75) node[] {\footnotesize{\scalebox{1}{$\Spvek{1;1;1;k-1}$}}};

\draw [very thin] (9.5,3.75) node[] {\tiny{\scalebox{0.8}{$\Spvek{k-1;1;k-1;1}$}}};
\draw [very thin] (9.5,1.25) node[] {\tiny{\scalebox{0.8}{$\Spvek{1;k-1;1;1}$}}};
\draw [very thin] (10.5,1.25) node[] {\tiny{\scalebox{0.8}{$\Spvek{k-2;1;1;1}$}}};
\draw [very thin] (13.5,3.75) node[] {\footnotesize{\scalebox{1}{$\Spvek{1;k-1;k-1;1}$}}};
\draw [very thin] (14,1.25) node[] {\footnotesize{\scalebox{1}{$\Spvek{1;1;k-1;1}$}}};

\draw [very thin] (2.5,7.25) node[] {\footnotesize{\scalebox{1}{$\Spvek{1;k-1;1;k-1}$}}};
\draw [very thin] (2,9.75) node[] {\footnotesize{\scalebox{1}{$\Spvek{1;k-1;1;k-1}$}}};
\draw [very thin] (4.5,9.75) node[] {\tiny{\scalebox{0.8}{$\Spvek{k-2;1;k-1;1}$}}};
\draw [very thin] (6.5,9.75) node[] {\footnotesize{\scalebox{1}{$\Spvek{1;k-1;k-1;1}$}}};
\draw [very thin] (5.5,7.25) node[] {\tiny{\scalebox{0.8}{$\Spvek{k-2;1;1;k-1}$}}};
\draw [very thin] (7,7.25) node[] {\tiny{\scalebox{1}{$\Spvek{1;1;k-1;k-1}$}}};

\draw [very thin] (12.5,9.75) node[] {\footnotesize{\scalebox{1}{$\Spvek{1;k-1;1;k-1}$}}};
\draw [very thin] (16.5,9.75) node[] {\tiny{\scalebox{0.8}{$\Spvek{k-1;1;1;k-1}$}}};
\draw [very thin] (12,7.25) node[] {\footnotesize{\scalebox{1}{$\Spvek{1;1;k-1;k-1}$}}};
\draw [very thin] (15.5,7.25) node[] {\tiny{\scalebox{0.8}{$\Spvek{k-2;1;1;k-1}$}}};
\draw [very thin] (16.5,7.25) node[] {\tiny{\scalebox{0.8}{$\Spvek{1;k-1;1;k-1}$}}};

\end{tikzpicture}
\end{center}
\label{fig:fifthkp}
\end{figure}

This inequality implies 
$$f_d(k,\ldots,k) \le  13f_{d-2}(k,\ldots,k)+9f_{d-3}(k,\ldots,k)$$
which in turn implies $f_d(k,\ldots,k) \le x_0^dk$ where $x_0$ is the largest root of $x^3-13x-9,$ $x_0 \approx 3.91.$ In particular, this shows that $\gamma_k\le \beta_k\le x_0.$

For small values of $k$ the above inequality actually implies a somewhat stronger result, provided we take more care with the $k-1,k-2$ terms. E.g. for $k=3$ we get:
\begin{align*}f_d(3,\ldots,3) \le & 10f_d(1,1,2,2,3,\ldots,3)+9f_d(1,1,1,2,3,\ldots,3)+
3f_d(1,1,1,1,3,\ldots,3)\\
\le & (10\cdot 4+9 \cdot 2+ 3)f_{d-4}(3,\ldots,3)=61f_{d-4}(3,\ldots,3)
\end{align*}
Where we repeatedly used $f_d(2,a_1,\ldots,a_{d-1}) \le 2f_d(1,a_1,\ldots,a_{d-1})=f_{d-1}(a_1,\ldots,a_{d-1}),$ which follows by taking two identical copies of the $d-1$ dimensional example. This inequality implies $\gamma_3 \le \beta_3 \le \sqrt[4]{61}\approx 2.79.$

\subsubsection{Boxes}
It is highly unclear how one could use the additional freedom afforded by using boxes instead of bricks. The only ways that we found exploits the fact that it is possible to cover a square using only three boxes, as shown in Figure~\ref{fig:cover}. This allows us to obtain better examples using boxes than the ones using bricks described above.
\begin{figure}
\caption{A square can be covered by $3$ boxes, but not with $3$ bricks.}
\begin{center}
\begin{tikzpicture}
\draw [] (0,0) rectangle (4,4);
\draw [] (0,1) -- (4,1);
\draw [] (0,3) -- (4,3);
\draw [] (1,0) -- (1,4);
\draw [] (3,0) -- (3,4);
\draw[pattern=vertical lines, pattern color=red] (0,1) rectangle (3,4);
\draw[pattern=north east lines, pattern color=green] (0,0) rectangle (1,1);
\draw[pattern=north east lines, pattern color=green] (3,0) rectangle (4,1);
\draw[pattern=north east lines, pattern color=green] (0,3) rectangle (1,4);
\draw[pattern=north east lines, pattern color=green] (3,3) rectangle (4,4);
\draw[pattern=horizontal lines, pattern color=blue] (1,0) rectangle (4,3);
\end{tikzpicture}
\end{center}
\label{fig:cover}
\end{figure}
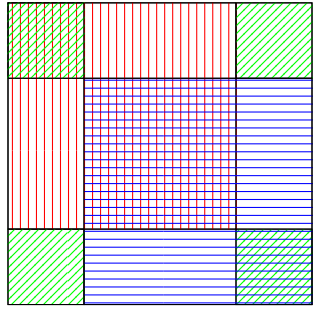

We will reuse the intermediate partition given in Figure~\ref{fig:secondkp} to obtain a new intermediate $3$ dimensional partition which will use three copies of it stacked on top of each other such that in each layer the copy of $A_r$ incident to vertex $Y$ is stretched to make one of the three boxes used to cover a square in Figure~\ref{fig:cover} and divided into $k-1$ copies of itself along the third dimension, as shown in Figure~\ref{fig:sixthkp}. This implies that
$f_d(k,\ldots,k) \le 9f_d(1,1,k-1,k,\ldots,k)+6f_d(1,1,k-2,k,\ldots,k).$

\begin{proof}[Proof of Theorem~\ref{thm:kpiercingbox}]
The above inequality directly implies $f_d(k,\ldots,k) \le 15 f_{d-2}(k,\ldots,k),$ showing $\gamma_k \le \sqrt{15} \approx 3.87$ implying Theorem~\ref{thm:kpiercingbox}. 
\end{proof}
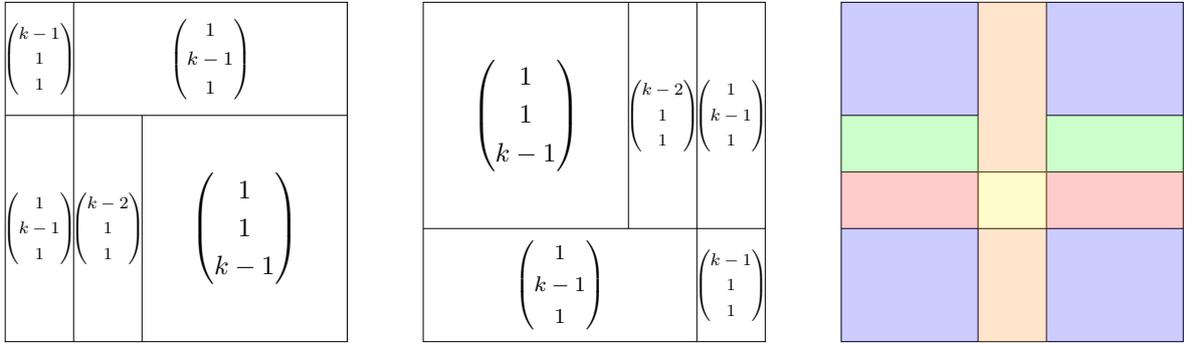
\begin{figure}
\caption{An  intermediate partition based on the above observation. The blue box is labelled $\protect\tvect{1}{1}{k-1}$, red and orange are labelled $\protect\tvect{k-1}{1}{1}$, green is labelled $\protect\tvect{1}{k-1}{1}$ and yellow is labelled $\protect\tvect{1}{k-1}{1}$. All other boxes are in fact bricks.}
\begin{center}
\begin{tikzpicture}
\draw [very thin] (0,0) rectangle (4.5,4.5);
\draw [very thin] (5.5,0) rectangle (10,4.5);
\draw [very thin] (11,0) rectangle (15.5,4.5);
\draw [very thin] (0.9,0) -- (0.9,4.5);
\draw [very thin] (0,3) -- (4.5,3);
\draw [very thin] (1.8,0) -- (1.8,3);
\draw [very thin] (5.5,1.5) -- (10,1.5);
\draw [very thin] (9.1,0) -- (9.1,4.5);
\draw [very thin] (8.2,1.5) -- (8.2,4.5);
\draw [very thin] (12.8,0) -- (12.8,4.5);
\draw [very thin] (13.7,0) -- (13.7,4.5);
\draw [very thin] (11,1.5) -- (15.5,1.5);
\draw [very thin] (13.7,3) -- (15.5,3);
\draw [very thin] (11,2.25) -- (15.5,2.25);
\draw [very thin] (11,3) -- (12.8,3););

\fill[blue,opacity=0.2] (11,0) rectangle (12.8,1.5);
\fill[blue,opacity=0.2] (11,3) rectangle (12.8,4.5);
\fill[blue,opacity=0.2] (13.7,0) rectangle (15.5,1.5);
\fill[blue,opacity=0.2] (13.7,3) rectangle (15.5,4.5);
\fill[red,opacity=0.2] (11,1.5) rectangle (12.8,2.25);
\fill[red,opacity=0.2] (13.7,1.5) rectangle (15.5,2.25);
\fill[green,opacity=0.2] (11,2.25) rectangle (12.8,3);
\fill[green,opacity=0.2] (13.7,2.25) rectangle (15.5,3);
\fill[yellow,opacity=0.2] (12.8,1.5) rectangle (13.7,2.25);
\fill[orange,opacity=0.2] (12.8,0) rectangle (13.7,1.5);
\fill[orange,opacity=0.2] (12.8,2.25) rectangle (13.7,4.5);

\draw [very thin] (0.45,3.75) node[] {\tiny{\scalebox{0.8}{$\Spvek{k-1;1;1}$}}};
\draw [very thin] (2.7,3.75) node[] {\tiny{\scalebox{0.9}{$\Spvek{1;k-1;1}$}}};
\draw [very thin] (0.45,1.5) node[] {\tiny{\scalebox{0.8}{$\Spvek{1;k-1;1}$}}};
\draw [very thin] (1.35,1.5) node[] {\tiny{\scalebox{0.8}{$\Spvek{k-2;1;1}$}}};
\draw [very thin] (3.15,1.5) node[] {\footnotesize{\scalebox{1}{$\Spvek{1;1;k-1}$}}};

\draw [very thin] (6.85,3) node[] {\footnotesize{\scalebox{1}{$\Spvek{1;1;k-1}$}}};
\draw [very thin] (7.3,0.75) node[] {\tiny{\scalebox{1}{$\Spvek{1;k-1;1}$}}};
\draw [very thin] (9.55,0.75) node[] {\tiny{\scalebox{0.8}{$\Spvek{k-1;1;1}$}}};
\draw [very thin] (9.55,3) node[] {\tiny{\scalebox{0.8}{$\Spvek{1;k-1;1}$}}};
\draw [very thin] (8.65,3) node[] {\tiny{\scalebox{0.8}{$\Spvek{k-2;1;1}$}}};

\end{tikzpicture}
\end{center}
\label{fig:sixthkp}
\end{figure}

\subsection{Lower bounds}\label{subseq:2dim}

The lower bound on $p_{\text{brick}}$ given in \eqref{eq:brick-trivial} seemed like having a good chance of actually being the truth. For example, it is tight for all $k$ in two dimensions, as the left image in Figure~\ref{fig:firstkp} shows that $p_{\text{brick}}(2,k) \le 4(k-1),$ which matches the lower bound. In higher dimensions it satisfies the recursive lower bound obtained by the inclusion-exclusion principle through analysing the number of bricks touching faces of dimensions from $0$ to $d-1$; for example the proof of the lower bound used only faces of dimensions $0$ (corners) and $1$ (edges). It turns out however that $21 \le p_{\text{brick}}(3,3)$ showing that the bound is not always tight. In fact, exploiting this fact and the aforementioned inclusion-exclusion inequality one can obtain a lower bound, for $k=3,$ which is by a constant factor better than \eqref{eq:brick-trivial}. We omit further details as both parts of the argument are quite cumbersome and result in only a very weak improvement. 




The case of boxes seems more difficult, even in $2$ dimensions. We conjecture that $p_{\text{box}}(2,k) = 4(k-1)$ ($=p_{\text{brick}}(2,k)$) and show that this is in fact asymptotically correct, as $k$ gets large. To this end we consider the following reduction. Given a partition of $[n]^2$ with the $k$-piercing property we construct an auxiliary graph with one vertex for each box. We colour the edge between two vertices red if there exists a vertical line intersecting both boxes and blue if there exists a horizontal line intersecting both boxes. Since  the $k$-piercing constraint requires that any line intersects at least $k$ boxes, we see that every vertex in our auxiliary graph is both contained inside a clique of at least $k$ vertices with all edges coloured red and a clique of at least $k$ vertices with all edges coloured blue. We therefore formulate the following question, which we find interesting in its own right.

\begin{question}\label{qu:graphqn}
Let $k\ge 1$ be an integer. What is the minimal $N$ such that we can colour the edges of a graph on $N$ vertices red and blue such that every vertex belongs to a monochromatic $K_k$ of each colour?
\end{question}

Note that, by the above construction, the answer to the above question is an upper bound for $p_{\text{box}}(2,k).$ We conjecture that this $N=4(k-1)$. A construction arising from the example in the left image of Figure~\ref{fig:firstkp} which matches this bound can be seen in Figure~\ref{fig:graph}. However, we were only able to prove an asymptotic result.

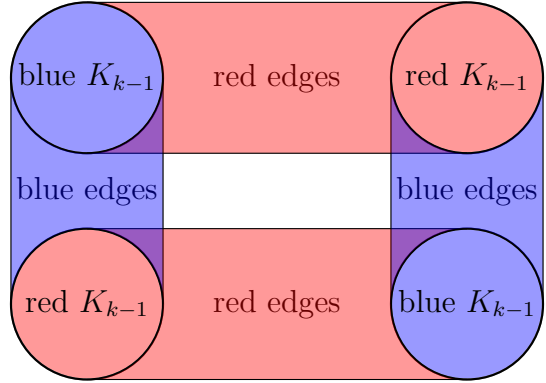
\begin{figure}
\caption{A graph in which every vertex is contained in a red $K_k$ and a blue $K_k$.}
\label{fig:graph}
\begin{center}
\begin{tikzpicture}
\draw [-,name path = u] (0,-1) -- (5,-1);
\draw [-,name path = v] (0,1) -- (5,1);
\draw [-, name path = w] (0,2) -- (5,2);
\draw [-, name path = x] (0,4) -- (5,4);
\draw [-, name path = w] (-1,0) -- (-1,3);
\draw [-, name path = x] (1,0) -- (1,3);
\draw [-, name path = w] (4,0) -- (4,3);
\draw [-, name path = x] (6,0) -- (6,3);
\draw[thick] (0cm,0cm) circle(1cm)[fill=red,opacity=0.4];
\draw[thick] (0cm,3cm) circle(1cm)[fill=blue,opacity=0.4];
\draw[thick] (5cm,0cm) circle(1cm)[fill=blue,opacity=0.4];
\draw[thick] (5cm,3cm) circle(1cm)[fill=red,opacity=0.4];
\draw[thick, name path = A] (0cm,0cm) circle(1cm)[];
\draw[thick, name path = B] (0cm,3cm) circle(1cm)[];
\draw[thick, name path = C] (5cm,0cm) circle(1cm)[];
\draw[thick, name path = D] (5cm,3cm) circle(1cm)[];
\draw [very thin] (0,0) node[] {red $K_{k-1}$};
\draw [very thin] (5,3) node[] {red $K_{k-1}$};
\draw [very thin] (5,0) node[] {blue $K_{k-1}$};
\draw [very thin] (0,3) node[] {blue $K_{k-1}$};
\draw [very thin] (2.5,0) node[] {red edges};
\draw [very thin] (2.5,3) node[] {red edges};
\draw [very thin] (0,1.5) node[] {blue edges};
\draw [very thin] (5,1.5) node[] {blue edges};
\draw[fill=red,opacity=0.4]
  ([shift={(-90:1cm)}]0,0) arc (-90:90:1cm)
  --
  ([shift={(-90+180:1cm)}]5,0) arc (-90+180:90+180:1cm)
  -- cycle;
\draw[fill=red,opacity=0.4]
    ([shift={(-90:1cm)}]0,3) arc (-90:90:1cm)
    --
    ([shift={(-90+180:1cm)}]5,3) arc (-90+180:90+180:1cm)
    -- cycle;
\draw[fill=blue,opacity=0.4]
        ([shift={(0:1cm)}]0,0) arc (0:180:1cm)
        --
        ([shift={(-180:1cm)}]0,3) arc (-180:0:1cm)
        -- cycle;
\draw[fill=blue,opacity=0.4]
        ([shift={(0:1cm)}]5,0) arc (0:180:1cm)
        --
        ([shift={(-180:1cm)}]5,3) arc (-180:0:1cm)
        -- cycle;

\end{tikzpicture}
\end{center}
\end{figure}

\begin{proposition}\label{prop:graph}
In Question~\ref{qu:graphqn} we have $N\ge (4+o_k(1))k$.
\end{proposition}

\begin{proof}
Let $R$ be the vertex set of a largest red clique and $B$ the vertex set of a largest blue clique in the graph. Note that $|R \cap B| \le 1,$ as each edge can only have one colour. Define $A_0 =R \setminus B$ and $B_0=B \setminus R.$ Let $a_0=|A_0| \ge k-1$ and $b_0=|B_0|\ge k-1.$

In general, let $R$ and $B$ be the vertex sets of a largest red and blue clique on $G \setminus ( A_0 \cup \ldots \cup A_{i-1} \cup B_0 \cup \ldots \cup B_{i-1}),$ respectively. As before, $|R \cap B| \le 1$ and we define $A_i =R \setminus B$ and $B_i=B \setminus R.$ Let $a_i=|A_i|$ and $b_i=|B_i|.$

Given a vertex $v$ in $A_0 \cup \ldots \cup A_{i-1}$ it belongs to a blue $k$-clique. This clique can have at most one vertex in each of $A_0,A_1, \ldots,A_{i-1},$ one of which is $v$ itself. Similarly, by choice of $B_i$ we know this clique can have at most $b_i+1$ vertices outside of $A_0 \cup \ldots \cup A_{i-1} \cup B_0 \cup \ldots \cup B_{i-1}.$ This implies that $v$ has blue degree at least $k-1-(i-1)-(b_i+1)=k-i-b_i-1$ towards $B_0 \cup \ldots \cup B_{i-1}.$ An analogous argument shows that any $w \in B_0 \cup \ldots \cup B_{i-1}$ has red degree at least $k-i-a_i-1$ towards $A_0 \cup \ldots \cup A_{i-1}.$

In particular, letting $A=a_0+\ldots+a_{i-1}$ and $B=b_0+\ldots+b_{i-1}$, this implies that
$$AB \ge A(k-i-b_i-1)+B(k-i-a_i-1).$$
Now define $c_{i-1}$ by $A+B=c_{i-1}(k-1).$ Since there are at least $c_{i-1}(k-1)+a_i+b_i$ vertices in $G.$ 
So we get $$AB+Ab_i+Ba_i \ge (k-i-1)c_{i-1}(k-1).$$

For a fixed $c_{i-1}$ the left-hand side is maximised for $A=c_{i-1}(k-1)/2-(a_i-b_i)/2$ and $B=c_{i-1}(k-1)/2+(a_i-b_i)/2.$ This gives

$$c_{i-1}^2(k-1)^2/4-(a_i-b_i)^2/4+(a_i+b_i)c_{i-1}(k-1)/2+(a_i-b_i)^2/2 \ge (k-i-1)c_{i-1}(k-1)$$

$$\Rightarrow c_{i-1}^2(k-1)^2+(a_i-b_i)^2+2(a_i+b_i)c_{i-1}(k-1)\ge 4(k-i-1)c_{i-1}(k-1)$$

$$\Rightarrow c_{i-1}^2(k-1)^2+(a_i+b_i)^2+2(a_i+b_i)c_{i-1}(k-1)\ge 4(k-i-1)c_{i-1}(k-1)$$

$$\Rightarrow (c_{i-1}(k-1)+a_i+b_i)^2\ge 4(k-i-1)c_{i-1}(k-1).$$

Since $c_i(k-1)=c_{i-1}(k-1)+a_i+b_i,$ we get 

$$c_i \ge 2 \sqrt{\frac{k-i-1}{k-1}c_{i-1}}\ge 2^{1+1/2+\ldots+1/2^{i}}\left(\frac{k-i-1}{k-1} \right)^{1/2+1/4+\ldots+1/2^{i}}$$
$$=4 \times 2^{-1/2^{i+1}}(1-i/(k-1))^{1-1/2^{i+1}}.$$

Choosing $i=\mathcal{O}(\log(k))$ gives the result.
\end{proof}

Proposition~\ref{prop:2dim} follows immediately from this result, by the above reduction. 

Note that Question~\ref{qu:graphqn} generalises naturally to $t$ colours. The proof of Proposition~\ref{prop:graph} can be easily modified to give a lower bound of $(2t+o_k(1))k$ for this generalisation, and the construction on Figure~\ref{fig:graph} can also be modified to give an upper bound of $2t(k-1)$. While the lower bound for this question applies to the $k$-piercing question, giving a lower bound of $(2d+o_d(1))k$ in $d$-dimensions which does beat the trivial bound of $d(k-1)$ from the start of the section, this bound is not particularly strong so we omit the full details. It seems that in two dimensions Question~\ref{qu:graphqn} captures the difficulty of the $k$-piercing problem, while the generalised version does not fully capture the difficulties of the higher dimensional piercing problem. 

With this in mind we consider the following reduction. Given a $k$-piercing partition in $d$ dimensions, consider the complete graph $K_n$ with vertices being boxes. We colour an edge between two boxes in colour $i$ if they are intersected by some $d-1$ dimensional plane orthogonal to the $i$-th dimensional axis. This gives a colouring in $d$ colours, such that every edge gets at most $d-1$ colours. 
Furthermore, every vertex is a part of a monochromatic $K_t$ in each colour, where $t=p_{\text{box}}(d-1,k).$ We shall use this to give the following lower bound.
\begin{theorem}
$$p_{\text{box}}(d,k) \ge e^{\frac{\sqrt{d}}{4}}(k-1)$$
\end{theorem}
\begin{proof}
We consider the complement of the colouring of the $K_n$ described in the previous paragraph. In the complement each edge gets assigned only the colours it was not assigned in the above colouring. As each edge had at most $d-1$ colours, the new colouring assigns at least one colour to each edge. Furthermore, for every vertex $v$ and every colour $c$, $v$ belongs to a set of size $t$ within which there is no edge of colour $c.$ 

We claim that this implies that for each colour there are at most $(n-t)^2$ edges of this colour. To see this, note that there needs to exist an independent set of size $t$ in this colour and each of the remaining $n-t$ vertices can be incident to at most $n-t$ edges of this colour. 

As our new colouring needed to cover all the possible edges at least once, this implies that 
\begin{align*}
d & \ge\frac{n(n-1)}{2(n-t)^2} \\
\implies n-1 & \ge \left(1+\frac{1}{\sqrt{2d}-1}\right)(t-1) \\
\implies p_{\text{box}}(d,k)-1 & \ge \left(1+\frac{1}{\sqrt{2d}-1}\right)(p_{\text{box}}(d-1,k)-1).
\end{align*}
This gives
\begin{align*}
p_{\text{box}}(d,k) & \ge \prod_{i=2}^{d}\left(1+\frac{1}{\sqrt{2i}-1}\right)(k-1)+1 \\
& \ge e^{\sum_{i=2}^{d} \frac{1}{2\sqrt{2i}}}(k-1) \\
& \ge e^{\frac{1}{2\sqrt{2}}\sum_{i=2}^{d} \frac{1}{\sqrt{i}}}(k-1) \\
& \ge e^{\frac{\sqrt{d}}{4}}(k-1)
\end{align*}
as claimed.

\end{proof}
\section{Conclusion and open problems}\label{sec:concl}
There are a large number of very interesting questions that remain in this area, and we shall now list just a few.

It remains, of course, to determine the asymptotics of $\fo$. The most important question seems to be the following.

\begin{question}\label{qu:finalfo}
Is $\fo(n,d)=(2+o(1))^d$ as $n,d\rightarrow \infty$?
\end{question}

One may also consider the original question of Kearnes and Kiss with a relaxation of the condition that the boxes partition $[n]^d$. In their paper~\cite{leader}, Leader, Mili\'{c}evi\'{c} and Tan ask how many proper boxes are required to form a double cover of $[n]^d$, and specifically whether at least $2^d$ are required. A natural construction involves taking three copies of a partition of $[n]^{d-1}$ and taking the products of these with the sets $\{1,2\}$, $\{2,\dots,n\}$ and $\{1,3,4,\ldots,n\}$ respectively, giving a double cover of size $(3/2)2^d$. We can show that this construction is not best possible (a simulated annealing approach found a double cover of size 11 in $[3]^3$ and Gurobi did even better by finding a construction of size 21 in $[3]^4$), but we have not been able to beat $2^d$ and the question remains open. 

Regarding the $k$-piercing problem, there are several possible angles. Again, the most important question concerns improving the lower bound.

\begin{question}\label{qu:finalkp}
Does there exist an $\varepsilon>0$ such that for a fixed $k$ we have $p_{\text{box}}(d,k) \ge (2+\varepsilon)^d$?
\end{question}

The analogous question for $p_{\text{brick}}$ would be a natural first step, interesting in its own right. 

Along similar lines is the regime where $d$ is fixed and $k$ is allowed to grow. As discussed in Section~\ref{sec:piercing}, the bound for this problem is always linear in $k$, but finding the constant of linearity seems hard. 

\begin{question}\label{qu:finalkp2}
Let $d$ be fixed so that $p_{\text{box}}(d,k)=(C_d+o_k(1))k$. How does $C_d$ grow with $d$? Must $C_d$ be exponential in $d$? 
\end{question}

As noted in Section~\ref{sec:piercing}, we are only able to show that $ e^{\frac{\sqrt{d}}{4}}(k-1)\le C_d\le 15^{d/2}$. Proposition~\ref{prop:2dim} shows that $C_2=4$, but finding $C_3$ is already beyond our methods. Answering this question would directly extend Theorem~\ref{thm:alon} and therefore probably requires some interesting new ideas.


To finish, we shall describe one last problem which is of particular interest. We observe that in the $k$-piercing problem the requirement that the boxes $B_i$ partition $[n]^d$ can be dropped without trivialising the question, provided that we maintain the constraint that the $B_i$ are disjoint. In particular, we could ask the following question.

\begin{question}\label{qu:nopartition}
Let $n\ge k$ and $d\ge 1$ be integers. Let $\{B^1,B^2,\dots, B^m\}$ be a collection of disjoint proper boxes in $[n]^d$ with $k$-piercing property. What lower bounds can be shown for $m$? In particular, do we have $m\ge 2^d$?
\end{question}

When $k=2$ this generalises the original question of Kearnes and Kiss, however the proof of Theorem~\ref{thm:alon} relies on the $B_i$ forming a partition and so the same idea cannot be used. Indeed the authors know of no approach that gives a bound better than $(1+o(1))^d$ for this question, although computer search finds no examples with $m<2^d$.

\section*{Acknowledgements}
We thank Imre Leader and Bhargav Narayanan for useful conversations about this project.

\vspace{-0.2cm}
\section{Appendix}
\subsection{List of coordinates of boxes in Figure~\ref{fig:25oddbox}}
~

\noindent\texttt{Box(1) = \{1,2,3\} x \{1,2,3\} x \{1\}\\
Box(2) = \{1,2,3\} x \{1,2,3\} x \{2\}\\
Box(3) = \{2,4,5\} x \{1,4,5\} x \{3\}\\
Box(4) = \{2,3,5\} x \{2,3,5\} x \{4\}\\
Box(5) = \{1,2,4\} x \{1,2,4\} x \{5\}\\
Box(6) = \{1,2,5\} x \{1\} x \{4\}\\
Box(7) = \{1\} x \{1,2,5\} x \{3\}\\
Box(8) = \{1\} x \{2,4,5\} x \{4\}\\
Box(9) = \{2,4,5\} x \{2\} x \{3\}\\
Box(10) = \{2,4,5\} x \{3\} x \{3\}\\
Box(11) = \{2,3,4\} x \{3\} x \{5\}\\
Box(12) = \{3\} x \{2,3,4\} x \{3\}\\
Box(13) = \{3\} x \{2,4,5\} x \{5\}\\
Box(14) = \{4\} x \{1,2,3\} x \{1,2,4\}\\
Box(15) = \{5\} x \{1,2,3\} x \{1,2,5\}\\
Box(16) = \{2,4,5\} x \{4\} x \{1,2,4\}\\
Box(17) = \{2,4,5\} x \{5\} x \{1,2,5\}\\
Box(18) = \{1\} x \{4\} x \{1,2,3\}\\
Box(19) = \{1\} x \{5\} x \{1,2,5\}\\
Box(20) = \{3\} x \{4\} x \{1,2,4\}\\
Box(21) = \{3\} x \{5\} x \{1,2,3\}\\
Box(22) = \{1\} x \{3\} x \{3,4,5\}\\
Box(23) = \{3\} x \{1\} x \{3,4,5\}\\
Box(24) = \{4\} x \{5\} x \{4\}\\
Box(25) = \{5\} x \{4\} x \{5\}\\
}
\end{document}